\documentclass[leqno]{amsart}
\usepackage{amsmath,amsthm,amsfonts,amssymb}

\numberwithin{equation}{section}
\theoremstyle{plain}

\newtheorem{theorem}{Theorem}

\newtheorem{corollary}[theorem]{Corollary}

\newtheorem{example}[theorem]{Example}
\newtheorem{lemma}[theorem]{Lemma}

\newtheorem{remark}[theorem]{Remark}

\newtheorem{assumption}[theorem]{Assumption}

\newtheorem*{theorem*}{Theorem}

\begin{document}

\title[Moderate and Large Deviations for the Erd\H{o}s-Kac Theorem]{Moderate and Large Deviations for the Erd\H{o}s-Kac Theorem}

\author{Behzad Mehrdad}
\address
{Courant Institute of Mathematical Sciences\newline
\indent New York University\newline
\indent 251 Mercer Street\newline
\indent New York, NY-10012\newline
\indent United States of America}
\email{mehrdad@cims.nyu.edu}

\author{Lingjiong Zhu}
\address
{School of Mathematics\newline
\indent University of Minnesota-Twin Cities\newline
\indent 206 Church Street S.E.\newline
\indent Minneapolis, MN-55455\newline
\indent United States of America}
\email{ling@cims.nyu.edu}

\date{16 May 2015. \textit{Revised:} 24 November 2015}

\subjclass[2000]{60F10,11N37.}
\keywords{The Erd\H{o}s-Kac theorem, additive functions, number of distinct prime factors, large deviations, moderate deviations.}

\begin{abstract}
The Erd\H{o}s-Kac theorem is a celebrated result in number theory which says that the number of distinct prime factors
of a uniformly chosen random integer satisfies a central limit theorem. In this paper, we establish the large deviations
and moderate deviations for this problem in a very general setting for a wide class of additive functions.
\end{abstract}

\maketitle

\section{Introduction}

Let $V(n)$ be a random integer chosen uniformly
from $\{1,2,\ldots,n\}$ and let $X(n)$ be the number of the distinct primes
in the factorization of $V(n)$. A celebrated result by Erd\H{o}s and Kac \cite{ErdosII}, \cite{ErdosIII} says that
\begin{equation}
\frac{X(n)-\log\log n}{\sqrt{\log\log n}}\rightarrow N(0,1),
\end{equation}
as $n\rightarrow\infty$. This is a deep extension of Hardy-Ramanujan Theorem
(see \cite{Hardy}).

In addition, the central limit theorem holds in a more general setting for additive functions.
A formal treatment and proofs can be found in e.g. Durrett \cite{Durrett}.

In terms of rate of convergence to the Gaussian distribution, i.e. Berry-Esseen bounds,
R\'{e}nyi and Tur\'{a}n \cite{Renyi} obtained the sharp rate of convergence
$O(1/\sqrt{\log\log n})$. Their proof is based on the analytic theory of Dirichlet series.
Recently, Harper \cite{Harper} used a more probabilistic approach and used Stein's method
to get an upper bound of rate of convergence of the order $O(\log\log\log n/\sqrt{\log\log n})$.

In terms of large deviations, Radziwill \cite{Radziwill} used analytic number theory approach
to get a series of asymptotic estimates. 
F\'{e}ray et al. \cite{Feray} proved precise large deviations using the mod-Poisson convergence
method developed in Kowalski and Nikeghbali \cite{Kowalski}.

In this paper, we study the large deviation principle
and moderate deviation principle in the sense of the Erd\H{o}s and Kac.
Instead of precise deviations that have been studied in F\'{e}ray et al. \cite{Feray},
our large deviations and moderate deviations results are in the sense of Donsker-Varadhan \cite{VaradhanI}, \cite{DonskerI},
\cite{DonskerII}, \cite{DonskerIII}, \cite{DonskerIV}.
Our proofs are probabilistic and require only elementary number theory results,
in contrast to the analytical number theory approach in Radziwill \cite{Radziwill}.
Since in this paper, we are considering the Donsker-Varadhan type large deviations, 
we only obtain the leading order term for the tail estimates, in contrast to the more precise estimates
in the works by Radziwill \cite{Radziwill}
and F\'{e}ray et al. \cite{Feray}. Because of this, 
we need much weaker assumptions on the additive functions in our paper.
Both large deviations and moderate deviations in our paper 
will be proved for a much wider class of
additive functions than what have been studied in Radziwill \cite{Radziwill} and F\'{e}ray et al. \cite{Feray}.
It is also worth mentioning that the large deviations theory from probability theory has been recently applied to study other problems
in number theory, see e.g. \cite{MZ, Zhu,FangI, FangII,Hu}. 
It might have the potential to become a useful tool in analytic number theory.

Before we proceed, let us introduce the formal definition of large deviations.  
A sequence $(P_{n})_{n\in\mathbb{N}}$ of probability measures on a topological space $X$ 
satisfies the large deviation principle with speed $b_{n}$ and rate function $I:X\rightarrow\mathbb{R}\cup\{\infty\}$ 
if $I$ is non-negative, 
lower semicontinuous and for any measurable set $A$, we have
\begin{equation}
-\inf_{x\in A^{o}}I(x)\leq\liminf_{n\rightarrow\infty}\frac{1}{b_{n}}\log P_{n}(A)
\leq\limsup_{n\rightarrow\infty}\frac{1}{b_{n}}\log P_{n}(A)\leq-\inf_{x\in\overline{A}}I(x).
\end{equation}
Here, $A^{o}$ is the interior of $A$ and $\overline{A}$ is its closure. 
We refer to Dembo and Zeitouni \cite{Dembo} or Varadhan \cite{VaradhanII} for general background of large deviations and the applications.

Let $X_{1},\ldots,X_{n}$ be a sequence of $\mathbb{R}^{d}$-valued i.i.d. random vectors with mean $0$ and convariance matrix $C$
that is invertible. Assume that $\mathbb{E}[e^{\langle\theta,X_{1}\rangle}]<\infty$, for $\theta$ in some ball around
the origin. For any sequence $a_{n}$ so that $\frac{\sqrt{n}}{a_{n}}$
and $\frac{a_{n}}{n}\rightarrow 0$ as $n\rightarrow\infty$,
a moderate deviation principle says that for any Borel set $A$,
\begin{align}
-\frac{1}{2}\inf_{x\in A^{o}}\langle x,C^{-1}x\rangle
&\leq\liminf_{n\rightarrow\infty}\frac{n}{a_{n}^{2}}\log\mathbb{P}\left(\frac{1}{a_{n}}\sum_{i=1}^{n}X_{i}\in A\right)
\\
&\leq\limsup_{n\rightarrow\infty}\frac{n}{a_{n}^{2}}\log\mathbb{P}\left(\frac{1}{a_{n}}\sum_{i=1}^{n}X_{i}\in A\right)
\leq-\frac{1}{2}\inf_{x\in A^{o}}\langle x,C^{-1}x\rangle.\nonumber
\end{align}
In other words, $\mathbb{P}(\frac{1}{a_{n}}\sum_{i=1}^{n}X_{i}\in\cdot)$ satisfies a large deviation principle with the speed $\frac{a_{n}^{2}}{n}$.
The above classical result can be found for example in \cite{Dembo}.
Moderate deviation principle fills in the gap between central limit theorem and large deviation principle.

In this paper, we are interested to prove both large deviations and moderate deviations for the Erd\H{o}s-Kac theorem
for a wide class of additive functions.

\section{Main Results}

Throughout this paper, $p$ always denotes a prime number.

\begin{assumption}\label{MainAssumption}
Let $g$ be a strongly additive, i.e. $g(p^{k})=g(p)$ for all primes $p$ and integers $k\geq 1$,
and $g(mn)=g(m)+g(n)$ whenever $\text{gcd}(m,n)=1$. In addition, we assume that
there exists a probability measure $\rho$ on $\mathbb{R}$ so that
\begin{itemize}
\item
For any $\theta\in\mathbb{R}$, $\int_{\mathbb{R}}e^{\theta y}\rho(dy)<\infty$.

\item
For any $\theta\in\mathbb{R}$, $\int_{\mathbb{R}}e^{\theta y}\rho_{n}(dy)\rightarrow\int_{\mathbb{R}}e^{\theta y}\rho(dy)$,
where
\begin{equation}
\rho_{n}(A)=\frac{\sum_{g(p)\in A,p\leq n}\frac{1}{p}}{\sum_{p\leq n}\frac{1}{p}},
\end{equation}
for any Borel set $A\subset\mathbb{R}$.
\end{itemize}
\end{assumption}

Let $V(n)$ be a uniformly chosen random integer from $\{1,2,\ldots,n\}$ and $Z_{p}=1$ if $V(n)$
is divisible by $p$ and $Z_{p}=0$ otherwise. Then, for any strongly additive function $g$, we have
$g(V(n))=\sum_{p\leq n}g(p)Z_{p}$. 
Let $X(n):=g(V(n))$.
We have the following large deviations result.

\begin{theorem}\label{LDPThm}
Under Assumption \ref{MainAssumption}, $\mathbb{P}(\frac{X(n)}{\log\log n}\in\cdot)$ satisfies a large deviation
principle with speed $\log\log n$ and rate function
\begin{equation}
I(x):=\sup_{\theta\in\mathbb{R}}\left\{\theta x-\int_{\mathbb{R}}(e^{\theta y}-1)\rho(dy)\right\}.
\end{equation}
\end{theorem}

When $g(p)\rightarrow\lambda\in(0,\infty)$ as $p\rightarrow\infty$,
$\rho(dy)$ tends to a point mass at $\lambda$ and the supremum in $I(x)$
in Theorem \ref{LDPThm} is achieved at $\theta=\lambda^{-1}\log(x/\lambda)$,
which gives the following corollary.

\begin{corollary}
Assume that $g(p)\rightarrow\lambda\in(0,\infty)$ as $p\rightarrow\infty$.
$\mathbb{P}(\frac{X(n)}{\log\log n}\in\cdot)$ satisfies a large deviation
principle with speed $\log\log n$ and rate function
\begin{equation}
I(x):=
\begin{cases}
\frac{x}{\lambda}\log\frac{x}{\lambda}-\frac{x}{\lambda}+1 &\text{if $x\geq 0$},
\\
+\infty &\text{otherwise}.
\end{cases}
\end{equation}
For the special case $g(p)\equiv 1$, $X(n)$ denotes the number of distinct prime factors
of a random integer uniformly chosen from $\{1,2,\ldots,n\}$ and 
$\mathbb{P}(\frac{X(n)}{\log\log n}\in\cdot)$ satisfies a large deviation
principle with speed $\log\log n$ and rate function
\begin{equation}
I(x):=
\begin{cases}
x\log x-x+1 &\text{if $x\geq 0$},
\\
+\infty &\text{otherwise}.
\end{cases}
\end{equation}
\end{corollary}

\begin{remark}\label{FourRemarks}
(i) If Assumption \ref{MainAssumption} fails, Theorem \ref{LDPThm} may not hold.
For example, we can define
\begin{equation}
g(p)=
\begin{cases}
\lambda_{1} &\alpha_{2k}<p\leq\alpha_{2k+1}
\\
\lambda_{2} &\alpha_{2k+1}<p\leq\alpha_{2k+2}
\end{cases},
\end{equation}
where $0<\lambda_{1}<\lambda_{2}$ and $k=0,1,2,\ldots$ and $\alpha_{i}-\alpha_{i-1}$ are sufficiently
large so that for some small $\delta>0$, and some fixed $\theta>0$, we have
\begin{equation}
\frac{1}{\log\log\alpha_{n}}\sum_{p\leq\alpha_{n}}\frac{e^{\theta g(p)}-1}{p}\leq\lambda_{1}+\delta
\end{equation}
along the subsequence $\alpha_{n}$ when $n$ is odd and
\begin{equation}
\frac{1}{\log\log\alpha_{n}}\sum_{p\leq\alpha_{n}}\frac{e^{\theta g(p)}-1}{p}\geq\lambda_{2}-\delta
\end{equation}
along the subsequence $\alpha_{n}$ when $n$ is even. This shows that for any fixed $\theta>0$, the limit
$\frac{1}{\log\log n}\log\mathbb{E}[e^{\theta X(n)}]$ does not exist. Therefore there is no large deviation
principle as a result of Varadhan's lemma, see e.g. Dembo and Zeitouni \cite{Dembo}.

(ii) The assumption that for any $\theta\in\mathbb{R}$, $\int_{\mathbb{R}}e^{\theta y}\rho(dy)<\infty$ basically
says that the density of primes $p$ such that $g(p)$ is large should be small.  
For example, if $g(p)$ grows to infinity
as $p\rightarrow\infty$, then, the scaling of $\sum_{p\leq n}\frac{e^{\theta g(p)}-1}{p}$ will
depend on $\theta$ and thus there is no $f(n)\rightarrow\infty$ independent of $\theta$ so that
the limit $\frac{1}{f(n)}\sum_{p\leq n}\frac{e^{\theta g(p)}-1}{p}$ exists.
On the other hand, if $g(p)\rightarrow 0$ as $p\rightarrow\infty$. 
Then, $\frac{e^{\theta g(p)}-1}{p}\sim\frac{\theta g(p)}{p}$ as $p\rightarrow\infty$. 
Now, if also $\sum_{p\leq n}\frac{g(p)}{p}\rightarrow\infty$ 
as $n\rightarrow\infty$, then, we have
\begin{equation}
\lim_{n\rightarrow\infty}\frac{1}{\sum_{p\leq n}\frac{g(p)}{p}}\sum_{p\leq n}\frac{e^{\theta g(p)}-1}{p}\rightarrow\theta,
\end{equation}
and the rate function
for the large deviations is then trivial. Again, if $\sum_{p}\frac{g(p)}{p}<\infty$, there is no Donsker-Varadhan type
large deviations.

(iii) Let $p_{1}<p_{2}<p_{3}<\cdots$ be the ordered sequence of all the prime numbers. 
Prime number theorem implies that
$p_{k+1}-p_{k}\leq\frac{p_{k}}{(\log p_{k})^{2}}$
for sufficiently large $k$. This shows that 
\begin{equation}
\sum_{p_{k}\leq n}\left(\frac{1}{p_{k}}-\frac{1}{p_{k+1}}\right)
=\sum_{p_{k}\leq n}\frac{p_{k+1}-p_{k}}{p_{k}p_{k+1}}\leq C\sum_{p_{k}\leq n}\frac{1}{p_{k}(\log p_{k})^{2}}
\end{equation}
for some universal constant $C$ and the series is convergent. 
Therefore, we have $\sum_{p_{2k}\leq n}\frac{1}{p_{2k}}\sim\frac{1}{2}\log\log n$
and $\sum_{p_{2k+1}\leq n}\frac{1}{p_{2k+1}}\sim\frac{1}{2}\log\log n$.
Let $0<\lambda_{1}<\lambda_{2}<\infty$.
Define $g(p_{k})=\lambda_{1}$ if $k$ is odd and $g(p_{k})=\lambda_{2}$ if $k$ is even.
Thus, we have
\begin{equation}
\lim_{n\rightarrow\infty}\int_{\mathbb{R}}(e^{\theta y}-1)\rho_{n}(dy)
=\frac{1}{2}(e^{\theta\lambda_{1}}-1)+\frac{1}{2}(e^{\theta\lambda_{2}}-1).
\end{equation}
Hence, we conclude that $\mathbb{P}(\frac{X(n)}{\log\log n}\in\cdot)$ satisfies a large deviation
principle with speed $\log\log n$ and rate function
\begin{equation}
I(x):=
\begin{cases}
\sup_{\theta\in\mathbb{R}}\left\{\theta x-\frac{1}{2}(e^{\theta\lambda_{1}}-1)-\frac{1}{2}(e^{\theta\lambda_{2}}-1)\right\} &\text{if $x\geq 0$},
\\
+\infty &\text{otherwise}.
\end{cases}
\end{equation}
In some special cases, there is an explicit expression for the rate function.
For example, consider $\lambda_{2}=2\lambda_{1}\in(0,\infty)$. In this case, the optimal $\theta$ is given by
\begin{equation}
\theta_{\ast}=\frac{1}{\lambda_{1}}\log
\left(\frac{-\lambda_{1}+\sqrt{\lambda_{1}^{2}+16\lambda_{1}x}}{4\lambda_{1}}\right).
\end{equation}
Hence, for $x\geq 0$, 
\begin{align}\label{ExampleThree}
I(x)&=\theta_{\ast}x-\frac{1}{2}(e^{\theta_{\ast}\lambda_{1}}-1)-\frac{1}{2}(e^{\theta_{\ast}\lambda_{2}}-1)
\\
&=\frac{x}{\lambda_{1}}\log
\left(\frac{-\lambda_{1}+\sqrt{\lambda_{1}^{2}+16\lambda_{1}x}}{4\lambda_{1}}\right)
+1\nonumber
\\
&\qquad\qquad
-\frac{1}{2}\left(\frac{-\lambda_{1}+\sqrt{\lambda_{1}^{2}+16\lambda_{1}x}}{4\lambda_{1}}\right)
-\frac{1}{2}\left(\frac{-\lambda_{1}+\sqrt{\lambda_{1}^{2}+16\lambda_{1}x}}{4\lambda_{1}}\right)^{2}.
\nonumber
\end{align}
\end{remark}

Here are some examples in which we can get an explicit expression for the rate function $I(x)$.

\begin{example}
(i) Assume that $\rho(dy)=\frac{1}{2}\delta_{\lambda_{1}}+\frac{1}{2}\delta_{2\lambda_{1}}$.
Then, from Remark \ref{FourRemarks} (iii), for $x\geq 0$, the rate function $I(x)$ has explicit
expression as given in \eqref{ExampleThree}.

(ii) Assume that $\rho(dy)$ has Poisson distribution with parameter $\lambda>0$. Then, we have
\begin{equation}
\theta x-\int_{0}^{\infty}(e^{\theta y}-1)\rho(dy)
=\theta x+1-e^{\lambda(e^{\theta}-1)}.
\end{equation}
Differentiating with respect to $\theta$ and setting the derivative as zero, we get
$x=\lambda e^{\theta}e^{\lambda(e^{\theta}-1)}$. 
Then, we get the optimal $\theta_{\ast}=\log(\frac{1}{\lambda}W(xe^{\lambda}))$,
where $W(z)$ is the Lambert W function defined as $z=W(z)e^{W(z)}$ for any $z\in\mathbb{C}$, see e.g. Corless et al. \cite{Corless}.
Hence, for $x\geq 0$, we have
\begin{equation}
I(x)=x\log\left(\frac{1}{\lambda}W(xe^{\lambda})\right)+1-e^{W(xe^{\lambda})-\lambda}.
\end{equation}

(iii) Assume that $\rho(dy)$ has binomial distribution with parameters $n$ and $\beta$.
Then, we get $\int_{0}^{\infty}e^{\theta y}\rho(dy)=(1-\beta+\beta e^{\theta})^{n}$. Thus, the optimal $\theta_{\ast}$
satisfies $x=n(1-\beta+\beta e^{\theta_{\ast}})^{n-1}\beta e^{\theta_{\ast}}$.
For $n=1$, $I(x)=x\log(x/\beta)+\beta-x$. For $n=2$, 
\begin{equation}
I(x)=x\log\left(\frac{-(1-\beta)+\sqrt{(1-\beta)^{2}+2x}}{2\beta}\right)
-\left(\frac{1-\beta+\sqrt{(1-\beta)^{2}+2x}}{2}\right)^{2}.
\end{equation}

(iv) Assume that $\rho(dy)$ has Gaussian distribution with mean $0$ and variance $1$. Then, we get
$\int_{-\infty}^{\infty}e^{\theta y}\rho(dy)=e^{\frac{1}{2}\theta^{2}}$. Thus, the optimal $\theta_{\ast}$
satisfies $x=\theta_{\ast}e^{\frac{1}{2}\theta_{\ast}^{2}}$, which is equivalent to $x^{2}=\theta_{\ast}^{2}e^{\theta_{\ast}^{2}}$.
Therefore, we have $\theta_{\ast}=\sqrt{W(x^{2})}$, where $W(\cdot)$ is the Lambert W function and
\begin{equation}
I(x)=\sqrt{W(x^{2})}x+1-\frac{x}{\sqrt{W(x^{2})}}.
\end{equation}
\end{example}

We conclude this section by stating a moderate rate deviation principle,
which fills in the gap between central limit theorem and large deviation principle.

\begin{theorem}\label{MDPThm}
Let $\mu_{n}:=\sum_{p\leq n}\frac{g(p)}{p}$
and $\sigma_{n}^{2}:=\sum_{p\leq n}\frac{g(p)^{2}}{p}$. 
Let $a_{n}$ be a positive sequence so that 
$\frac{\sigma_{n}}{a_{n}}\rightarrow 0$ and $\frac{a_{n}}{\sigma_{n}^{2}}\rightarrow 0$
as $n\rightarrow\infty$.
Under Assumption \ref{MainAssumption},
$\mathbb{P}(\frac{X(n)-\mu_{n}}{a_{n}}\in\cdot)$ satisfies a large deviation principle 
with speed $a_{n}^{2}/\sigma_{n}^{2}$ and rate function
\begin{equation}
J(x):=\frac{x^{2}}{2}.
\end{equation}
\end{theorem}

\section{Proofs}

\subsection{Proof of Large Deviation Principle}

In this section, we will prove Theorem \ref{LDPThm}. The proof consists of a series of superexponential estimates, 
i.e. Lemma \ref{LargeC}, Lemma \ref{LargePrimes} and Lemma \ref{ZYLemma},
and an application of G\"{a}rtner-Ellis theorem (see e.g. Chapter 2 in Dembo and Zeitouni \cite{Dembo}).
For the convenience of the readers, we state here a special case of the G\"{a}rtner-Ellis theorem (Theorem 2.3.6 \cite{Dembo})
that we will use in the proofs in our paper.

\begin{theorem*}[G\"{a}rtner-Ellis Theorem \cite{Dembo}]
Let $Z_{n}$ be a sequence of random variables on $\mathbb{R}$ and $a_{n}$
is a positive sequence that goes to infinity as $n$ goes to infinity. Assume that
for any $\theta\in\mathbb{R}$, the limit $\Lambda(\theta):=\lim_{n\rightarrow\infty}\frac{1}{a_{n}}\log\mathbb{E}[e^{\theta a_{n}Z_{n}}]$
exists and is differentiable for every $\theta\in\mathbb{R}$. Then, $\mathbb{P}(Z_{n}\in\cdot)$ satisfies
a large deviation principle with rate function $I(x):=\sup_{\theta\in\mathbb{R}}\{\theta x-\Lambda(\theta)\}$
with speed $a_{n}$.
\end{theorem*}

Write $X(n)=\sum_{p\leq n}g(p)Z_{p}$, where $Z_{p}=1$ if $V(n)$ is divisible by $p$
and $0$ otherwise. The first step is to show that $\sum_{p\leq n}g(p)Z_{p}$
can be approximated by $\sum_{|g(p)|\leq C}g(p)Z_{p}$ in the following sense.

\begin{lemma}\label{LargeC}
for any $\epsilon>0$,
\begin{equation}\label{SuperC}
\limsup_{C\rightarrow\infty}\limsup_{n\rightarrow\infty}
\frac{1}{\log\log n}\log\mathbb{P}
\left(\left|\sum_{|g(p)|>C,p\leq n}g(p)Z_{p}\right|\geq\epsilon\log\log n\right)=-\infty.
\end{equation}
\end{lemma}

\begin{proof}
Note that \eqref{SuperC} holds if we can prove the following two estimates,
\begin{equation}\label{SuperCPositive}
\limsup_{C\rightarrow\infty}\limsup_{n\rightarrow\infty}
\frac{1}{\log\log n}\log\mathbb{P}
\left(\sum_{g(p)>C,p\leq n}g(p)Z_{p}\geq\epsilon\log\log n\right)=-\infty,
\end{equation}
and
\begin{equation}\label{SuperCNegative}
\limsup_{C\rightarrow\infty}\limsup_{n\rightarrow\infty}
\frac{1}{\log\log n}\log\mathbb{P}
\left(\sum_{g(p)<-C,p\leq n}g(p)Z_{p}\leq-\epsilon\log\log n\right)=-\infty.
\end{equation}

Before we proceed, let us define independent random variables $Y_{p}$, so that
\begin{equation}
Y_{p}=
\begin{cases}
1 &\text{with probability $\frac{1}{p}$},
\\
0 &\text{with probability $1-\frac{1}{p}$}.
\end{cases}
\end{equation}
For distinct primes $p_{1},p_{2},\ldots,p_{\ell}$,
\begin{equation}
\mathbb{E}[Z_{p_{1}}Z_{p_{2}}\cdots Z_{p_{\ell}}]
=\frac{1}{n}\bigg\lfloor\frac{n}{p_{1}p_{2}\cdots p_{\ell}}\bigg\rfloor
\leq\frac{1}{p_{1}p_{2}\cdots p_{\ell}}
=\mathbb{E}[Y_{p_{1}}Y_{p_{2}}\cdots Y_{p_{\ell}}],
\end{equation}
where $\lfloor x\rfloor$ denotes the largest integer less than or equal to $x$.
Therefore, for any non-negative sequence $\theta_{p}$, by Taylor's expansion, we have
\begin{equation}
\mathbb{E}\left[e^{\sum_{p\leq n}\theta_{p}Z_{p}}\right]
\leq\mathbb{E}\left[e^{\sum_{p\leq n}\theta_{p}Y_{p}}\right].
\end{equation}
This fact will be used repeatedly later on in the paper.

For $g(p)>C>0$, for any $\theta>0$, by Chebychev's inequality, we have $g(p)\theta>0$ and
\begin{align}
&\limsup_{n\rightarrow\infty}
\frac{1}{\log\log n}\log\mathbb{P}
\left(\sum_{g(p)>C,p\leq n}g(p)Z_{p}\geq\epsilon\log\log n\right)
\\
&\leq\limsup_{n\rightarrow\infty}\frac{1}{\log\log n}
\log\mathbb{E}\left[e^{\theta\sum_{g(p)>C,p\leq n}g(p)Z_{p}}\right]-\theta\epsilon
\nonumber
\\
&\leq
\limsup_{n\rightarrow\infty}\frac{1}{\log\log n}
\log\mathbb{E}\left[e^{\theta\sum_{g(p)>C,p\leq n}g(p)Y_{p}}\right]-\theta\epsilon
\nonumber
\\
&=\limsup_{n\rightarrow\infty}\frac{\sum_{g(p)>C, p\leq n}\log(\frac{e^{\theta g(p)}-1}{p}+1)}{\sum_{p\leq n}\frac{1}{p}}
-\theta\epsilon
\nonumber
\\
&\leq
\limsup_{n\rightarrow\infty}\frac{\sum_{g(p)>C, p\leq n}\frac{e^{\theta g(p)}-1}{p}}{\sum_{p\leq n}\frac{1}{p}}
-\theta\epsilon
\nonumber
\\
&=\limsup_{n\rightarrow\infty}\int_{y>C}(e^{\theta y}-1)\rho_{n}(dy)-\theta\epsilon
\nonumber
\\
&=\int_{y>C}(e^{\theta y}-1)\rho(dy)-\theta\epsilon,\nonumber
\end{align}
and the limit goes to $-\theta\epsilon$ as $C\rightarrow\infty$. By letting $\theta\rightarrow\infty$, 
we obtained \eqref{SuperCPositive}. Similarly, for $g(p)<-C<0$, by taking $\theta<0$, we get
\begin{align}
&\limsup_{n\rightarrow\infty}
\frac{1}{\log\log n}\log\mathbb{P}
\left(\sum_{g(p)<-C,p\leq n}g(p)Z_{p}\leq-\epsilon\log\log n\right)
\\
&\leq\int_{y<-C}(e^{\theta y}-1)\rho(dy)+\theta\epsilon,\nonumber
\end{align}
where the limit goes to $\theta\epsilon$ as $C\rightarrow\infty$. By letting $\theta\rightarrow-\infty$, 
we obtain \eqref{SuperCNegative}.
\end{proof}

Let 
\begin{equation}
k_{n}:=n^{\frac{1}{(\log\log n)^{2}}}.
\end{equation}
The second step is to show that $\sum_{p\leq n,|g(p)|\leq C}g(p)Z_{p}$
can be approximated by $\sum_{p\leq k_{n},|g(p)|\leq C}g(p)Z_{p}$ in the following sense.

\begin{lemma}\label{LargePrimes}
For any $\epsilon>0$,
\begin{equation}\label{superI}
\limsup_{n\rightarrow\infty}\frac{1}{\log\log n}\log\mathbb{P}
\left(\left|\sum_{p\in A(n,C)}g(p)Z_{p}\right|\geq\epsilon\log\log n\right)=-\infty,
\end{equation}
where $A(n,C):=\{p:k_{n}\leq p\leq n,|g(p)|\leq C\}$.
\end{lemma}

\begin{proof}
For any $\theta>0$,
\begin{align}
\mathbb{E}\left[e^{\theta|\sum_{p\in A(n,C)}g(p)Z_{p}|}\right]
&\leq\mathbb{E}\left[e^{\theta\sum_{p\in A(n,C)}|g(p)|Z_{p}}\right]
\\
&\leq\mathbb{E}\left[e^{\theta C\sum_{p\in A(n,C)}Z_{p}}\right]
\nonumber
\\
&\leq\mathbb{E}\left[e^{\theta C\sum_{p\in A(n,C)}Y_{p}}\right].
\nonumber
\end{align}
Therefore, for any $\theta>0$, we have
\begin{align}
\log\mathbb{E}\left[e^{\theta|\sum_{p\in A(n,C)}g(p)Z_{p}|}\right]
&\leq\log\mathbb{E}\left[e^{\theta C\sum_{p\in A(n,C)}Y_{p}}\right]
\\
&=\sum_{p\in A(n,C)}\log\left((e^{\theta C}-1)\frac{1}{p}+1\right)
\nonumber
\\
&\leq(e^{\theta C}-1)\sum_{k_{n}\leq p\leq n}\frac{1}{p}.
\nonumber
\end{align}
We also have
\begin{align}
\sum_{k_{n}\leq p\leq n}\frac{1}{p}
&=\log\log n-\log\log k_{n}+o(1)
\\
&=\log\log n-\log\left(\frac{1}{(\log\log n)^{2}}\log n\right)+o(1)
\nonumber
\\
&=2\log\log\log n+o(1),\nonumber
\end{align}
where we used the fact that $\sum_{p\leq n}\frac{1}{p}-\log\log n$ converges to the Meissel-Mertens constant.
Notice that $\frac{2\log\log\log n}{\log\log n}\rightarrow 0$ as $n\rightarrow\infty$. Therefore, by Chebychev's inequality,
\begin{equation}
\limsup_{n\rightarrow\infty}\frac{1}{\log\log n}\log\mathbb{P}
\left(\left|\sum_{p\in A(n,C)}g(p)Z_{p}\right|\geq\epsilon\log\log n\right)
\leq-\epsilon\theta.
\end{equation}
This proves \eqref{superI} since it holds for any $\theta>0$.
\end{proof}

Next, let us show that $\mathbb{E}\left[e^{\theta\sum_{p\leq k_{n},|g(p)|\leq C}g(p)Z_{p}}\right]$
can be approximated by $\mathbb{E}\left[e^{\theta\sum_{p\leq k_{n},|g(p)|\leq C}g(p)Y_{p}}\right]$ 
in an appropriate way.

\begin{lemma}\label{ZYLemma}
For any $\theta\in\mathbb{R}$,
\begin{equation}
\lim_{n\rightarrow\infty}\frac{1}{\log\log n}\log\left|\mathbb{E}\left[e^{\theta\sum_{p\in B(n,C)}g(p)Z_{p}}\right]
-\mathbb{E}\left[e^{\theta\sum_{p\in B(n,C)}g(p)Y_{p}}\right]\right|=-\infty,
\end{equation}
where $B(n,C):=\{p:p\leq k_{n},|g(p)|\leq C\}$.
\end{lemma}

\begin{proof}
For any $\theta\in\mathbb{R}$, and any
$K$ (which will later be chosen sufficiently large in terms of $\theta$ and $C$),
\begin{align}
&\left|\mathbb{E}\left[e^{\theta\sum_{p\in B(n,C)}g(p)Z_{p}}\right]
-\mathbb{E}\left[e^{\theta\sum_{p\in B(n,C)}g(p)Y_{p}}\right]\right|\label{ThreeTerms}
\\
&\leq\sum_{r\leq K\log\log n}\frac{|\theta|^{r}}{r!}|\mathbb{E}[S_{n}^{r}]-\mathbb{E}[\tilde{S}_{n}^{r}]|
+\sum_{r>K\log\log n}\frac{|\theta|^{r}}{r!}\mathbb{E}[S_{n}^{r}]
+\sum_{r>K\log\log n}\frac{|\theta|^{r}}{r!}\mathbb{E}[\tilde{S}_{n}^{r}],
\nonumber
\end{align}
where
\begin{equation}
S_{n}:=\sum_{p\in B(n,C)}g(p)Z_{p},
\qquad
\tilde{S}_{n}:=\sum_{p\in B(n,C)}g(p)Y_{p}.
\end{equation}

We claim that $|\mathbb{E}[\tilde{S}_{n}^{r}]-\mathbb{E}[S_{n}^{r}]|\leq\frac{(Ck_{n})^{r}}{n}$.
To see this, notice that
\begin{equation}
\mathbb{E}[\tilde{S}_{n}^{r}]=\sum_{k=1}^{r}\sum_{r_{i}}\frac{r!}{r_{1}!\cdots r_{k}!}\frac{1}{k!}
\sum_{p_{j}}g(p_{1})^{r_{1}}\cdots g(p_{k})^{r_{k}}\mathbb{E}[Y_{p_{1}}^{r_{1}}\cdots Y_{p_{k}}^{r_{k}}].
\end{equation}
We observe that
\begin{equation}
\mathbb{E}[Y_{p_{1}}^{r_{1}}\cdots Y_{p_{k}}^{r_{k}}]=\mathbb{E}[Y_{p_{1}}\cdots Y_{p_{k}}]
=\frac{1}{p_{1}\cdots p_{k}},
\end{equation}
which differs from
\begin{equation}
\mathbb{E}[Z_{p_{1}}^{r_{1}}\cdots Z_{p_{k}}^{r_{k}}]=\mathbb{E}[Z_{p_{1}}\cdots Z_{p_{k}}]
=\frac{1}{n}\bigg\lfloor\frac{n}{p_{1}\cdots p_{k}}\bigg\rfloor,
\end{equation}
by at most $\frac{1}{n}$. Therefore,
\begin{align}
|\mathbb{E}[\tilde{S}_{n}^{r}]-\mathbb{E}[S_{n}^{r}]|
&\leq\sum_{k=1}^{r}\sum_{r_{i}}\frac{r!}{r_{1}!\cdots r_{k}!}\frac{1}{k!}\sum_{p_{j}}\frac{C^{r_{1}+\cdots+r_{k}}}{n}
\\
&\leq\frac{1}{n}\left(\sum_{p\in B(n,C)}C\right)^{r}
\leq\frac{(Ck_{n})^{r}}{n}.
\nonumber
\end{align}

Thus, we can bound the first term in \eqref{ThreeTerms} by
\begin{align}
&\sum_{r\leq K\log\log n}\frac{|\theta|^{r}}{r!}|\mathbb{E}[S_{n}^{r}]-\mathbb{E}[\tilde{S}_{n}^{r}]|
\\
&\leq\sum_{r\leq K\log\log n}\frac{|\theta|^{r}}{r!}\frac{(Ck_{n})^{r}}{n}
\nonumber
\\
&\leq C_{0}e^{-\log n}\sum_{r\leq K\log\log n}e^{r(\log|\theta|+\log C+\log k_{n}+1)-r\log r}
\nonumber
\\
&\leq C_{0}e^{-\log n}K\log\log n\cdot\max_{r\leq K\log\log n}e^{r(\log|\theta|+\log C+\log k_{n}+1)-r\log r}.
\nonumber
\end{align}
Let $F(r):=r(\log|\theta|+\log C+\log k_{n}+1)-r\log r$. Since $k_{n}=n^{\frac{1}{(\log\log n)^{2}}}$ and
$\frac{k_{n}}{r}\rightarrow\infty$ as $n\rightarrow\infty$, it is straightforward to compute that
\begin{equation}
F'(r)=(\log|\theta|+\log C+\log k_{n}+1)-\log r-1>0,
\end{equation}
for any $r\leq K\log\log n$ and it is true for $n$ sufficiently large in terms of $C,K,\theta$.
Therefore the maximum of $F(r)$, $r\leq K\log\log n$ is achieved at $K\log\log n$ and, hence,
\begin{align}
\sum_{r\leq K\log\log n}\frac{|\theta|^{r}}{r!}|\mathbb{E}[S_{n}^{r}]-\mathbb{E}[\tilde{S}_{n}^{r}]|
&\leq
C_{1}\log\log n\cdot e^{-\log n}\cdot e^{K\log\log n\frac{\log n}{(\log\log n)^{2}}}
\\
&\leq e^{-\frac{1}{2}\log n},\nonumber
\end{align}
for sufficiently large $n$.

The second term in \eqref{ThreeTerms} is bounded above by the third term.
\begin{equation}
\sum_{r>K\log\log n}\frac{|\theta|^{r}}{r!}\mathbb{E}[S_{n}^{r}]
\leq\sum_{r>K\log\log n}\frac{|\theta|^{r}}{r!}\mathbb{E}[\tilde{S}_{n}^{r}].
\end{equation}
To bound the third term in \eqref{ThreeTerms}, first observe that
\begin{equation}
\mathbb{E}[\tilde{S}_{n}^{r}]=\sum_{p_{1},p_{2},\ldots,p_{\ell}\leq k_{n}}g(p_{1})\cdots g(p_{\ell})\mathbb{E}[Y_{p_{1}}\cdots Y_{p_{\ell}}],
\end{equation}
where the sums are over primes $p_{1},\ldots, p_{\ell}\leq k_{n}$ that may not be distinct.
Notice that $k_{n}\leq n$ and for distinct $p_{1},\ldots,p_{\ell}$, we have
\begin{equation}
\mathbb{E}[Y_{p_{1}}\cdots Y_{p_{\ell}}]=\frac{1}{p_{1}\cdots p_{\ell}}.
\end{equation}
Therefore, it is not difficult to see that
\begin{align}
\mathbb{E}[\tilde{S}_{n}^{r}]
&\leq C^{r}\left[\left(\sum_{p\leq n}\frac{1}{p}\right)+\left(\sum_{p\leq n}\frac{1}{p}\right)^{2}
+\cdots+\left(\sum_{p\leq n}\frac{1}{p}\right)^{r}\right]
\\
&\leq rC^{r}\left(\sum_{p\leq n}\frac{1}{p}\right)^{r}
\nonumber.
\end{align}
For $r>K\log\log n$,
\begin{align}
\frac{|\theta|^{r}}{r!}\mathbb{E}[\tilde{S}_{n}^{r}]
&\leq C_{2} e^{r\log|\theta|+r\log C-r\log r+r+\log r+(\log\log\log n)r}
\\
&\leq C_{2} e^{(\log|\theta|+\log C+1-\log K)r+\log r}
\nonumber
\\
&\leq C_{2} e^{-\frac{1}{2}(\log K)r},
\nonumber
\end{align}
for sufficiently large $K$, where $C_{2}$ is a positive constant. Hence,
the third term in \eqref{ThreeTerms} is bounded above by
\begin{equation}
\sum_{r>K\log\log n}e^{-\frac{1}{2}(\log K)r}\leq C_{3}e^{-\frac{1}{2}\log K\log\log n},
\end{equation}
where $C_{3}$ is a positive constant.
The proof is completed by letting $K\rightarrow\infty$.
\end{proof}

Finally, we are ready to prove Theorem \ref{LDPThm}.

\begin{proof}[Proof of Theorem \ref{LDPThm}]
Recall that $B(n,c)=\{p:|g(p)|\leq C, p\leq k_{n}\}$.
For any $\theta\in\mathbb{R}$, we have
\begin{align}
\log\mathbb{E}\left[e^{\theta\sum_{p\in B(n,C)}g(p)Y_{p}}\right]
&=\log\prod_{p\in B(n,C)}\mathbb{E}\left[e^{\theta g(p)Y_{p}}\right]
\\
&=\sum_{p\in B(n,C)}\log\left(\frac{1}{p}e^{\theta g(p)}+1-\frac{1}{p}\right).\nonumber
\end{align}
For sufficiently large $p$, 
\begin{equation}
\log\left(\frac{1}{p}e^{\theta g(p)}+1-\frac{1}{p}\right)
+(e^{\theta g(p)}-1)\frac{1}{p}=O\left(\frac{1}{p^{2}}\right),
\end{equation}
where $O(\frac{1}{p^{2}})$ depends on $C$ and $\theta$,
and it is well known that $\frac{1}{\log\log n}\sum_{p\leq n}\frac{1}{p}\rightarrow 1$ as $n\rightarrow\infty$,
and we also have
\begin{equation}
\lim_{n\rightarrow\infty}\frac{\log\log k_{n}}{\log\log n}=\lim_{n\rightarrow\infty}\frac{\log\log n-2\log\log\log n}{\log\log n}=1.
\end{equation}
Therefore, by Assumption \ref{MainAssumption} and definition of $\rho_{n}$ and $\rho$ 
\begin{equation}
\lim_{n\rightarrow\infty}\frac{1}{\log\log n}\log\mathbb{E}\left[e^{\theta\sum_{p\in B(n,C)}g(p)Y_{p}}\right]
=\int_{-C}^{C}(e^{\theta y}-1)\rho(dy),
\end{equation}
if $\rho(\{-C\})=\rho(\{C\})=0$ since convergence of moment generating functions implies
weak convergence. Here we can assume that $\rho(\{-C\})=\rho(\{C\})=0$ since there are at most countably many atoms for $\rho$, we can assume
that we choose a sequence of $C$ that goes to infinity so that $-C$ and $C$ are not atoms of $\rho$.

By lemma \ref{ZYLemma}, we have
\begin{align}
\lim_{n\rightarrow\infty}\frac{1}{\log\log n}\log\mathbb{E}\left[e^{\theta\sum_{p\in B(n,C)}g(p)Z_{p}}\right]
&=\lim_{n\rightarrow\infty}\frac{1}{\log\log n}\log\mathbb{E}\left[e^{\theta\sum_{p\in B(n,C)}g(p)Y_{p}}\right]
\\
&=\int_{-C}^{C}(e^{\theta y}-1)\rho(dy).\nonumber
\end{align}

By G\"{a}rtner-Ellis theorem, see e.g. Dembo and Zeitouni \cite{Dembo},
$\mathbb{P}(\frac{\sum_{p\in B(n,C)}g(p)Z_{p}}{\log\log n}\in\cdot)$ satisfies a large deviation principle
with the rate function 
\begin{equation}
I_{C}(x)=\sup_{\theta\in\mathbb{R}}
\left\{\theta x-\int_{-C}^{C}(e^{\theta y}-1)\rho(dy)\right\}.
\end{equation}

By the approximation estimates developed in Lemma \ref{LargeC} and Lemma \ref{LargePrimes}, 
the truncated errors are superexponentially small. Hence,
$\mathbb{P}(\frac{X(n)}{\log\log n}\in\cdot)$ satisfies a large deviation principle with the rate function 
\begin{equation}
I(x)=\lim_{C\rightarrow\infty}I_{C}(x)
=\sup_{\theta\in\mathbb{R}}
\left\{\theta x-\int_{-\infty}^{\infty}(e^{\theta y}-1)\rho(dy)\right\}.
\end{equation}
\end{proof}

\subsection{Proof of Moderate Deviation Principle}

We conclude this section by giving a proof of the moderate deviation principle.

\begin{proof}[Proof of Theorem \ref{MDPThm}]
Recall that $Y_{p}$ are independent Bernoulli random variables with parameter $\frac{1}{p}$ 
and $Z_{p}=1$ if $V(n)$ is divisible by $p$ and $Z_{p}=0$ otherwise, where $V(n)$ is
an integer uniformly distributed on $\{1,2,\ldots,n\}$.

Let 
\begin{equation}
k_{n}:=n^{\frac{a_{n}}{\sigma_{n}^{2}}}.
\end{equation}
Similar to the proof of Lemma \ref{LargeC}, we can show that, for any $\epsilon>0$,
\begin{equation}
\limsup_{C\rightarrow\infty}\limsup_{n\rightarrow\infty}
\frac{\sigma_{n}^{2}}{a_{n}^{2}}\log
\mathbb{P}\left(\left|\sum_{|g(p)|>C, p\leq n}g(p)Z_{p}\right|\geq\epsilon a_{n}\right)=-\infty,
\end{equation}
and also similar to the proof of Lemma \ref{LargePrimes}, we can show that, for any $\epsilon>0$,
\begin{equation}
\limsup_{n\rightarrow\infty}\frac{\sigma_{n}^{2}}{a_{n}^{2}}\log
\mathbb{P}\left(\sum_{k_{n}<p\leq n,|g(p)|\leq C}g(p)Z_{p}\geq\epsilon a_{n}\right)=-\infty,
\end{equation}
where we used the facts that $\frac{a_{n}}{\sigma_{n}^{2}}\rightarrow 0$
and $\frac{\sigma_{n}}{a_{n}}\rightarrow 0$.

Let us define
\begin{equation}
S_{n}=\sum_{p\in B(n,C)}g(p)Z_{p},
\quad\text{and}\quad
\tilde{S}_{n}=\sum_{p\in B(n,C)}g(p)Y_{p},
\end{equation}
where we recall that $B(n,C)=\{p:p\leq k_{n},|g(p)|\leq C\}$.

Let 
\begin{equation}
\mu_{n}^{C}:=\mathbb{E}[\tilde{S}_{n}]=\sum_{p\in B(n,C)}\frac{g(p)}{p},
\end{equation}
and recall that
\begin{equation}
\sigma_{n}^{2}:=\sum_{ p\leq n}\frac{g(p)^{2}}{p}.
\end{equation}

Following the proof of Lemma \ref{ZYLemma}, for any $\theta\in\mathbb{R}$, we can also prove that
\begin{equation}
\lim_{n\rightarrow\infty}\frac{\sigma_{n}^{2}}{a_{n}^{2}}\log\left|
\mathbb{E}\left[e^{\frac{\theta a_{n}}{\sigma_{n}^{2}}(S_{n}-\mu_{n}^{C})}\right]
-\mathbb{E}\left[e^{\frac{\theta a_{n}}{\sigma_{n}^{2}}(\tilde{S}_{n}-\mu_{n}^{C})}\right]
\right|=-\infty.
\end{equation}

Finally, for any $\theta\in\mathbb{R}$,
\begin{align}
&\frac{\sigma_{n}^{2}}{a_{n}^{2}}\log\mathbb{E}\left[e^{\frac{\theta a_{n}}{\sigma_{n}^{2}}(\tilde{S}_{n}-\mu_{n}^{C})}\right]
\\
&=-\frac{\theta}{a_{n}}\mu_{n}^{C}
+\frac{\sigma_{n}^{2}}{a_{n}^{2}}\sum_{p\in B(n,C)}
\log\left[\left(e^{\frac{\theta a_{n}}{\sigma_{n}^{2}}g(p)}-1\right)\frac{1}{p}+1\right]
\nonumber
\\
&=
-\frac{\theta}{a_{n}}\mu_{n}^{C}
+\frac{\sigma_{n}^{2}}{a_{n}^{2}}\sum_{p\in B(n,C)}
\left(e^{\frac{\theta a_{n}}{\sigma_{n}^{2}}g(p)}-1\right)\frac{1}{p}+\frac{\sigma_{n}^{2}}{a_{n}^{2}}O\left(\frac{a_{n}}{\sigma_{n}^{2}}\right)
\nonumber
\\
&=
-\frac{\theta}{a_{n}}\mu_{n}^{C}
+\frac{\sigma_{n}^{2}}{a_{n}^{2}}\sum_{p\in B(n,C)}
\left(\frac{\theta a_{n}}{\sigma_{n}^{2}}g(p)+\frac{1}{2}\frac{\theta^{2} a_{n}^{2}}{\sigma_{n}^{4}}g(p)^{2}+
g(p)^{2}O\left(\left(\frac{a_{n}}{\sigma_{n}^{2}}\right)^{3}\right)\right)
\frac{1}{p}+O\left(\frac{1}{a_{n}}\right)
\nonumber
\\
&=\frac{\theta^{2}}{2}\frac{1}{\sigma_{n}^{2}}\sum_{p\in B(n,C)}\frac{g(p)^{2}}{p}+o(1),
\nonumber
\end{align}
where $C,\theta$ are being viewed as fixed while $n\rightarrow\infty$.
Therefore, by  Assumption \ref{MainAssumption},
\begin{equation}
\lim_{n\rightarrow\infty}\frac{\sigma_{n}^{2}}{a_{n}^{2}}\log\mathbb{E}\left[e^{\frac{\theta a_{n}}{\sigma_{n}^{2}}(\tilde{S}_{n}-\mu_{n}^{C})}\right]
=\frac{\theta^{2}}{2}\frac{\int_{-C}^{C}y^{2}\rho(dy)}{\int_{-\infty}^{\infty}y^{2}\rho(dy)}.
\end{equation}
By letting $C\rightarrow\infty$ and applying
G\"{a}rtner-Ellis theorem, we complete the proof.
\end{proof}

\section*{Acknowledgements}

The authors are very grateful to Professor S. R. S. Varadhan for helpful discussions and generous suggestions. 
The authors are extremely grateful to an anonymous referee for carefully reading the paper, and for the suggestions
and comments that have greatly improved the quality of the manuscript.

\end{document}